\documentclass[11pt]{amsart}
\usepackage{hyperref}
\usepackage[capitalize, nameinlink]{cleveref}
\crefname{theorem}{Theorem}{Theorems}
\crefname{proposition}{Proposition}{Propositions}
\crefname{observation}{Observation}{Observations}
\crefname{lemma}{Lemma}{Lemmas}
\crefname{claim}{Claim}{Claims}
\crefname{problem}{Problem}{Problems}
\crefname{conjecture}{Conjecture}{Conjectures}
\crefname{question}{Question}{Questions}
\crefname{example}{Example}{Examples}
\crefname{fact}{Fact}{Facts}

\usepackage{amssymb}
\usepackage[bbgreekl]{mathbbol}
\usepackage{graphicx}
\usepackage{ifthen}
\usepackage{pict2e}
\usepackage{xargs}
\usepackage{xspace}
\usepackage{xcolor}
\usepackage{pgf,tikz}
\usepackage{pgfplots}
\usepackage{environ}
\usepackage{caption}
\usetikzlibrary{arrows}
\usepackage{enumitem}
\usepackage{xifthen}
\usepackage{comment}
\usepackage{soul}
\usepackage{newtxmath}
\usepackage{comment}

\newcounter{dummy}
\makeatletter
\newcommand\myitem[1][]{\item[#1]\refstepcounter{dummy}\def\@currentlabel{#1}}
\makeatother

\makeatletter
\newsavebox{\measure@tikzpicture}
\NewEnviron{scaletikzpicturetowidth}[1]{%
	\def\tikz@width{#1}%
	\begin{lrbox}{\measure@tikzpicture}%
		\BODY
	\end{lrbox}%
	\pgfmathparse{#1/\wd\measure@tikzpicture}%
	\BODY
}
\makeatother

\DeclareSymbolFontAlphabet{\mathbb}{AMSb}

\newcommand{\thistheoremname}{}
\newtheorem*{genericthm*}{\thistheoremname}
\newenvironment{namedthm*}[1]
{\renewcommand{\thistheoremname}{#1}%
	\begin{genericthm*}}
	{\end{genericthm*}}

\newcommand{\Bairespace}[1][]{
	\ifthenelse{\equal{#1}{}}{\functions{\N}{\N}}{\functions{#1}{\N}}
}
\newcommand{\bbL}{\mathbb{L}}
\newcommand{\bbX}{\mathbb{X}}

\newcommand{\Cantorspace}[1][]{
	\ifthenelse{\equal{#1}{}}{\functions{\N}{2}}{\functions{#1}{2}}
}

\newcommandx{\concatenation}[2][1 = undefined, 2 = undefined]{
	\ifthenelse{\equal{#1}{undefined}}{{}\smallfrown}{
		\ifthenelse{\equal{#2}{undefined}}{\bigoplus #1}{\bigoplus_{#1} #2}
	}
}

\newcommandx{\functions}[3][3 =]{
	\ifthenelse{\equal{#3}{}}{#2^{#1}}{#2_{#3}^{#1}}
}

\newcommand{\Gzero}[1][]{
	\ifthenelse{\equal{#1}{}}
	{\mathbb{G}_0}
	{\mathbb{G}_{0,n}}
}
\newcommandx{\Hzero}[2][2 = undefined]{
	\ifthenelse{\equal{#2}{undefined}}
	{\mathbb{H}_{#1}}
	{\mathbb{H}_{#1, #2}}
}
\newcommandx{\intersection}[2][1 =, 2 =]{
	\ifthenelse{\equal{#1}{}}{\cap}{
		\ifthenelse{\equal{#2}{}}{\bigcap #1}{{\bigcap_{#1} #2}}
	}
}
\newcommand{\Lzero}[1][]{\ifthenelse{\equal{#1}{}}{\bbL_0}{L_{0, #1}}}
\newcommand{\Lzerospace}[1][]{\ifthenelse{\equal{#1}{}}{\bbX_0}{X_{0, #1}}}

\newcommand{\modulo}[1]{\ (\text{mod } 2)}
\newcommand{\N}{\mathbb{N}}

\newcommandx{\product}[2][1 =, 2 =]{
	\ifthenelse{\equal{#1}{}}{\times}{
		\ifthenelse{\equal{#2}{}}{\prod #1}{{\prod_{#1} #2}}
	}
}

\newcommandx{\sequence}[2][2 = undefined]{
	\ifthenelse{\equal{#2}{undefined}}{(#1)}{
		(#1)_{#2}
	}
}

\newcommandx{\set}[2][2 = undefined]{
	\ifthenelse{\equal{#2}{undefined}}{\{ #1 \}}{
		\{ #1 \suchthat #2 \}
	}
}
\newcommandx{\sets}[3][3 =]{
	\ifthenelse{\equal{#3}{}}{[#2]^{#1}}{[#2]^{#1}_{#3}}
}

\newcommand{\suchthat}{\mid}

\renewcommand{\restriction}[2]{#1 \upharpoonright #2}

\newcommandx{\union}[2][1 =, 2 =]{
	\ifthenelse{\equal{#1}{}}{\cup}{
		\ifthenelse{\equal{#2}{}}{\bigcup #1}{{\bigcup_{#1} #2}}
	}
}

\newcommand{\homo}{Homo}
\newcommand{\inv}{Inv}

\newcommand{\stab}{\mathrm{Stab}}

\newtheorem{theorem}{Theorem}[section]
\newtheorem{lemma}[theorem]{Lemma}

\newtheorem{claim}[theorem]{Claim}
\newtheorem{corollary}[theorem]{Corollary}
\newtheorem{proposition}[theorem]{Proposition}

\newtheorem{problem}[theorem]{Problem}

\theoremstyle{definition}

\newtheorem{definition}[theorem]{Definition}

\newtheorem{remark}[theorem]{Remark}

\numberwithin{equation}{section}

\newcommand{\om}{\mathbb{N}}

\newcommand{\bd}{\begin{definition}}
	\newcommand{\ed}{\end{definition}}

\DeclareMathOperator{\dom}{dom}

\DeclareMathOperator{\dist}{dist}
\DeclareMathOperator{\didistance}{didist}
\newcommand{\mc}{\mathcal}

\newcommand{\distance}[3]{\ifthenelse{\isempty{#3}}{\dist(#1,#2)}{\dist^{#3}(#1,#2)}}
\newcommand{\didist}[3]{\ifthenelse{\isempty{#3}}{\didistance(#1,#2)}{\didistance^{#3}(#1,#2)}}
\newcommand{\digraph}[3]{\ifthenelse{\equal{#1}{b}}{\mathbb{#2}_{#3}}
	{{#2}_{#3}}}
\newcommand{\linegraph}[3]{\ifthenelse{\equal{#1}{b}}{\mathbb{#2}_{#3}}
	{#2_{#3}}}

\newcommand{\underlyingspace}[3]{\ifthenelse{\equal{#1}{b}}{\mathbb{#2}_{#3}}
	{#2_{#3}}}
\newcommand{\distanceset}[2]{\ifthenelse{\isempty{#2}}{D(#1)}{D^{#2}(#1)}}

\newcommand{\ovl}{\overline}

\newcommand{\cpct}[1]{K_{\mathcal{#1}}}
\newcommand{\cpcth}{K_{K_3}}

\newcommand{\wtilde}{\widetilde}

\newcommand{\fD}{\mathcal{D}}
\newcommand{\fE}{\mathcal{E}}

\newcommand{\fF}{\mathcal{F}}
\newcommand{\fH}{\mc{H}}

\newcommand{\fP}{\mathcal{P}}

\newcommand{\fU}{\mathcal{U}}

\newcommand{\fX}{\mathcal{X}}

\newtheorem{manualtheoreminner}{Theorem}
\newenvironment{manualtheorem}[1]{%
	\renewcommand\themanualtheoreminner{#1}%
	\manualtheoreminner
}{\endmanualtheoreminner}

\author{Tam\'as K\'atay}
\address{E\"otv\"os Lor\'and University, Institute of Mathematics, P\'azm\'any P\'eter stny. 1/C, 1117 Budapest, Hungary}
\email{13heted@gmail.com}

\author{L\'aszl\'o M\'arton T\'oth}
\address{Alfr\'ed R\'enyi Institute of Mathematics, Reáltanoda u. 13-15., 1053 Budapest, Hungary}
\email{toth.laszlo.marton@renyi.hu}

\author{Zolt\'an Vidny\'anszky}
\address{E\"otv\"os Lor\'and University, Institute of Mathematics, P\'azm\'any P\'eter stny. 1/C, 1117 Budapest, Hungary}
\email{zoltan.vidnyanszky@ttk.elte.hu}
\thanks{The first and third authors were supported by Hungarian Academy of Sciences Momentum Grant no. 2022-58 and National Research, Development and Innovation Office ( NKFIH) grants no.~113047, ~129211. The second author was supported by the ERC Consolidator Grant 772466 ''NOISE'', and the NKFIH grant KKP-139502, ''Groups and graph limits''.}
\begin{document}

	\thanks{}
	
	\keywords{CSP Dichotomy, Axiom of Choice, compactness, cyclic polymorphism}
	
	\subjclass[2020]{Primary 03E25, Secondary 68Q17}
	
	\title[]{The CSP Dichotomy, the Axiom of Choice, and Cyclic Polymorphisms}

	

	\begin{abstract}
    We study Constraint Satisfaction Problems (CSPs) with infinite instances and finite templates. We show that the dichotomy between easy and hard problems -- established by Bulatov and Zhuk in the finite case -- presents itself as the strength of the corresponding De Bruijn-Erdős-type compactness theorem over ZF.
		
    More precisely, if $\fD$ is a finite structure, let $K_\mc{D}$ stand for the following statement: for every structure $\mc{X}$ if every finite substructure of $\mc{X}$ admits a homomorphism to $\mc{D}$, then so does $\mc{X}$. We prove that if $\mc{D}$ admits no cyclic polymorphism, and thus the $\mc{D}$ homomorphism problem is NP-complete by the CSP Dichotomy Theorem, then $K_\mc{D}$ is equivalent to the Ultrafilter Lemma (UL) over ZF. Conversely, we also show that if $\mc{D}$ admits a cyclic polymorphism, and thus the $\mc{D}$ homomorphism problem is in P, then $K_\mc{D}$ is strictly weaker than UL.

    We do not use the CSP Dichotomy Theorem to obtain these results. Moreover, in our investigation the need to use cyclic polymorphisms arose naturally from the theory of choiceless models, suggesting a deeper connection between the finite and infinite cases. 
\end{abstract}
	
	\maketitle
	
	\section{Introduction}
	
	By \emph{structure} we mean a relational structure with finitely many relations. For structures $\fX$ and $\fD$ we say that $\fX$ is an \emph{instance of the $\fD$-homomorphism problem}, or, in short, an \emph{instance of $\mc{D}$} if their signature coincides. If $\mc{D}$ is a finite structure, the \emph{(finite) $\mc{D}$-homomorphism problem} is the problem to decide whether a given (finite) instance of $\mc{D}$ admits a homomorphism to $\mc{D}$. These homomorphism problems are also called \emph{Constraint Satisfaction Problems}  (CSPs). 
	
	Examples of CSPs include the problem of finding a proper $n$-coloring of a graph, graph homomorphism problems in general, but also solving various types of SAT problems or systems of linear equations over finite fields (see, e.g., \cite{feder1998computational}).
	
	It is easy to check that CSPs are always in $NP$, and after decades of concerted efforts toward understanding the computational complexity of CSPs, Bulatov \cite{bulatov2017dichotomy} and Zhuk \cite{zhuk2020proof} independently proved the remarkable CSP Dichotomy Theorem, conjectured by Feder and Vardi \cite{feder1998computational}. The separation between hard and easy problems relies on an algebraic property of $\mathcal{D}$, which we abbreviate by $(*)_\mathcal{D}$ (see Subsection \ref{ss:cyclic_polymorphisms} for the details). 
    
	\begin{theorem}[Bulatov, Zhuk]   
		\label{t:dichotomy}
		The $\mc{D}$-homomorphism problem is either in $P$ or $NP$-complete. More specifically:
		\begin{enumerate}
			\item if $(*)_\mathcal{D}$ holds, then the $\mc{D}$-homomorphism problem is in $P$;
			\item if $(*)_\mathcal{D}$ fails, then the $\mc{D}$-homomorphism problem is $NP$-complete.
		\end{enumerate}
	\end{theorem}
	
   It is rather natural to consider infinitary generalizations of the theorem above. Note however, that (at least from a certain perspective) there is no difference between deciding $2$-colorability and $3$-colorability of an infinite graph: by the compactness theorem, in both cases it suffices to check colorability of every finite subgraph. This suggests that one has to work in a context where the compactness theorem fails.
	
	In this paper, we investigate such a context, inspired by a remark of Halbeisen. Let $\cpct{D}$ stand for the following statement: for every (possibly infinite) instance $\mc{X}$ of the $\mc{D}$-homomorphism problem, if every finite substructure of $\mc{X}$ admits a homomorphism to $\mc{D}$, then so does $\mc{X}$. For example, if $\mc{D}=K_n$, the complete graph on $n$ vertices, then the statement $K_{K_n}$ is the De Bruijn-Erd\H os theorem.

    Clearly, $\cpct{D}$ follows from ZFC for all finite structures $\mc{D}$. However, if one drops the Axiom of Choice (AC), one can differentiate between $\cpct{D}$'s for various $\mc{D}$, expecting that $\cpct{D}$ is stronger when the $\mc{D}$-homomorphism problem is harder. We confirm that this is indeed the case, and it turns out that the right benchmark in this context is the Ultrafilter Lemma (UL), which states that every filter can be extended to an ultrafilter.
	
	\begin{theorem} (ZF)
		\label{t:main}
		\begin{enumerate}
			\item 
			\label{c:maineasy}
			If $(*)_\mc{D}$ holds, then $\cpct{D}$ is strictly weaker than UL. In fact, there is a model of ZF in which $\cpct{D}$ is true for exactly those $\mc{D}$'s for which $(*)_\mc{D}$ holds.
			\item 
			\label{c:mainhard}
			If $(*)_\mc{D}$ fails, then $\cpct{D} \iff UL$.
		\end{enumerate}
		
	\end{theorem}

   While UL is strictly weaker than AC, one can still vaguely interpret Theorem~\ref{t:main} as $\cpct{D}$ requiring much more choice when the homomorphism problem is complicated. In Subsection \ref{ss:hard} we will illustrate why the Ultrafilter Lemma fits the world of CSP's.
   
	Theorem~\ref{t:main} generalizes results of L\'evy, Mycielski, and L\"auchli \cite{levy1963remarks,mycielski1961some, lauchli1971coloring, cowen1990two}, and verifies a conjecture posed in \cite[(7) p. 238]{cowen1990two}. In \cite{rorabaugh2017logical} Rorabaugh, Tardif, and Wehlau also study the strength of $K_{\mc{D}}$ without AC. Case of (\ref{c:mainhard}) extends Proposition 4.3 from their paper, while in (\ref{c:maineasy}) we settle Problem~4.4.
	
	The key idea of our proof is to work with so called cyclic polymorphisms, investigated by Barto and Kozik \cite{barto2012absorbing} instead of other, more commonly used equivalent forms of $(*)_\mathcal{D}$. Surprisingly, the need to use cyclic polymorphisms has naturally arisen from the purely infinitary machinery of permutation models, and this strongly suggests that there is a deeper connection between choiceless models and CSPs than one would expect. Let us also emphasize that the theorems below do not use the breakthrough results of Bulatov and Zhuk, hence the following question is rather natural.
	
	\begin{problem}
		Is it possible to adapt some ideas from the proof of Theorem \ref{t:main} to give a novel proof of the CSP Dichotomy Theorem?
	\end{problem}
	
	\noindent{\sc Roadmap.} In Section \ref{s:detailed} we provide a brief history of the CSP Dichotomy Theorem and some of its possible infinitary generalizations. Then we give an overview of the proof ideas of our result and formulate a black box theorem, which can be used without familiarity with forcing or abstract set theory. Then, in Section \ref{s:hard} we discuss the proof of the second part of the main theorem with details postponed to the Appendix, while in Section \ref{s:easy} we apply and prove the black box theorem and the first part of the main result. Finally, we conclude the paper by listing a couple of open problems.

	\section{Context and proof overview}
	\label{s:detailed}
	
	\subsection{The CSP Dichotomy Theorem}
	It was observed early on that in every known case the $\mc{D}$-homomorphism problem is either in $P$ or $NP$-complete. Schaefer \cite{schaefer1978complexity} proved this to be the case when $D$ has size at most two\footnote{We follow the convention of using calligraphic letters to denote structures and plain upper case letters to denote their universe.}. Later, Hell and Neset\v{r}il \cite{hell1990complexity} gave a complete characterization when $\mc{D}$ is a graph: they showed that if $\mc{D}$ is bipartite, then the problem is in $P$, and otherwise it is $NP$-complete. 
	
	Based on the mounting evidence, Feder and Vardi \cite{feder1998computational} formulated the statement which became known as the CSP Dichotomy Conjecture: 
	\[\text{The $\mc{D}$-homomorphism problem is either in $P$ or $NP$-complete}.\]
	Note that by a classical theorem of Ladner \cite{ladner1975structure}, if $P \neq NP$ then there are problems in $NP$ which are neither in $P$ nor $NP$-complete.
	
	Since the formulation of the conjecture, an immense amount of effort has been put into trying to verify it (some of the outstanding results are \cite{maroti2008existence,kun2012linear,barto2009constraint,barto2012absorbing,barto2009csp}, see \cite{brady2022notes} for a comprehensive introduction). Bulatov, Jeavons and Krokhin \cite{bulatov2005classifying} isolated an algebraic condition that seemed to characterize the easy problems. They proved that if it fails, then the $\mc{D}$-homomorphism problem is $NP$-complete and conjectured that conversely, if it holds, then the problem is in $P$. Finally, Bulatov \cite{bulatov2017dichotomy} and Zhuk \cite{zhuk2020proof} independently proved the conjecture, see Theorem~\ref{t:dichotomy} above.
	
	\subsection{Cyclic polymorphisms}\label{ss:cyclic_polymorphisms}
	The key algebraic tool for the investigation of CSPs turned out to be polymorphisms. An \emph{($n$-ary) polymorphism} is a homomorphism $\phi: \mc{D}^n \to \mc{D}$. Here $\mc{D}^n$ is the categorical power of the structure $\mc{D}$, i.e.\ if $R$ is a $k$-ary relation of $\mc{D}$, then it is interpreted on $\mc{D}^n$ as follows: let $\overline{x}_i=\big(\overline{x}_i(1), \ldots, \overline{x}_i(n)\big) \in D^n$ for $i=1, \ldots,k$, then  $(\overline{x}_1, \ldots, \overline{x}_k) \in R^{\mc{D}^n} \iff \big(\overline{x}_1(j), \ldots \overline{x}_k(j)\big) \in R^{\mc{D}}$ for all $1\leq j\leq n$.
	A useful way to look at, say, $n$-ary polymorphisms is that they combine $n$-many homomorphisms to $\mc{D}$ into a new one. 
	
	Clearly, projection maps are always polymorphisms. There are several results which show that the existence of essentially non-projective polymorphisms imply the existence of ones of special forms, the most common of them are being so called weak near-unanimity (WNU) operations -- the Dichotomy Theorem is typically stated using those (\cite{bulatov2017dichotomy} and \cite{zhuk2020proof}). In our considerations, however the following type will play a crucial role.
	
	\begin{definition}			
		A polymorphism $\phi: \mc{D}^n \to \mc{D}$ is called \emph{cyclic} if $n\geq 2$ and $\phi$ satisfies \[\phi(x_0,x_1,\dots,x_{n-1})=\phi(x_1,x_2,\dots,x_{n-1},x_0)\]
		for all $(x_0,\dots,x_{n-1}) \in D^n$. 
	\end{definition}
 
     The following result is due to Barto and Kozik:

    \begin{theorem}\label{t:cyclic_polymorphisms}
    \cite{barto2012absorbing} For a finite relational structure $\fD$, the following are equivalent:
	\begin{enumerate}
		\item $\mathcal{D}$ admits a WNU. 
		\item $\mathcal{D}$ admits a cyclic polymorphism.
		\item $\mathcal{D}$ admits a cyclic polymorphism of every prime arity $p$ with $p>|D|$.
	\end{enumerate}

    \end{theorem}


			
	\begin{definition}
		\label{d:tractability}
		Let $(*)_\mc{D}$ denote that condition (3) holds in Theorem~\ref{t:cyclic_polymorphisms}.
	\end{definition}
	
	\subsection{Infinite versions of the CSP Dichotomy}
	
	Let us briefly mention that there are several infinite versions of the CSP dichotomy problem. A flourishing direction is to allow $\mc{D}$ to be infinite, while still requiring $\mc{X}$ to be finite and ask about the computational complexity of the problem. In this case a rich structure theory emerges, with a myriad of questions yet to be answered (see, e.g., \cite{bodirsky2021complexity,bodirsky2006constraint,bodirsky2015schaefer,bodirsky2015topological, barto2020topology, mottet2022smooth, pinsker2022dilemmas}). 
	
	Another direction, which has been recently initiated by Thornton \cite{riley} is to keep $\mc{D}$ finite, and require the homomorphisms and the instance to be Borel. An advantage of this approach is that in the Borel context, it can be proved that say, $2$-coloring is easier than $3$-coloring, see \cite{toden}. It seems however, that solving systems of linear equations over finite fields is already hard in the Borel context \cite{grebikv}, thus, the split between hard and easy problems occurs at a different place. 
	
	The direction explored in this paper is to keep $\fD$ finite and investigate the strength of the statement $K_{\fD}$ for infinite $\fX$. We build on the work of L\'evy, Mycielski, and L\"auchli \cite{levy1963remarks,mycielski1961some, lauchli1971coloring}, who showed that over ZF, for $n \geq 3$ the statement $K_{K_n}$ is equivalent to UL and $K_{K_2}$ is significantly weaker (see also \cite{howard1998consequences,halbeisen2012combinatorial}).
	
	Our Theorem~\ref{t:main} shows that the same split between easy and hard problems occurs in the infinite context as in the CSP Dichotomy. That is, $K_\mc{D}$ is equivalent to UL if $(*)_\mc{D}$ does not hold, while otherwise it is a weaker statement. Note also that in our case this is a dichotomy in the usual sense, unlike the finite case, where if $P=NP$ then the two cases collapse into one. 
	
	In the rest of the paper, we work over ZF, unless specified otherwise.

	\subsection{Hard problems}
    \label{ss:hard}


    In this subsection, we first sketch a short proof of a classical theorem of Läuchli and Lévy, in order to illustrate why the Ultrafilter Lemma (UL) is natural to consider in the context of infinite CSPs (see also the $k$-switch problem of Greenwell-Lov\'asz \cite{greenwell1974applications}, Komj\'ath-Totik \cite{komjath2008ultrafilters} and Rorabaugh-Tardif-Wehlau \cite[Lemma 4.1]{rorabaugh2017logical}). 
 
    \begin{theorem}\label{t:K_K_3_UL}
        \cite{lauchli1971coloring}
        The compactness principle $K_{K_3}$ is equivalent to the Ultrafilter Lemma.
    \end{theorem}

    \begin{proof}[Proof sketch.]
    It is easy to see that $UL$ implies the compactness theorem of first-order logic, which in turn implies $K_{K_3}$. 

    In order to see the converse, let $X$ be a set and $\fF\subseteq\fP(X)$ a filter on $X$. We will extend $\fF$ to an ultrafilter. Consider the following graph $G$ on $3^X$:
    $$(f,g)\in G\overset{\text{def}}{\iff}\{x\in X:\ f(x)\neq g(x)\}\in\fF.$$
	By $\mathcal{F}$ being a filter, it follows that every finite subgraph of $G$ admits a proper $3$-coloring. Hence, $K_{K_3}$ implies the existence of a coloring $c:3^X\to 3$. 
	
	The key idea is that from $c$ one can define an ultrafilter extending $\mathcal{F}$. Indeed, we can assume without loss of generality that $c$ maps the constant functions ($\vmathbb{0}_X$, $\vmathbb{1}_X$, and $\vmathbb{2}_X$) to their values $0$, $1$, and $2$ respectively. Then define
    $$\fU=\{A\subseteq X:\ c(\vmathbb{1}_A)=1\},$$
    where $\vmathbb{1}_A(x)=1$ if $x\in A$ and $0$ otherwise. Checking that $\fU$ is indeed an ultrafilter containing $\mathcal{F}$ is a nice, nontrivial exercise.  
\end{proof}
       


    In particular, if $X$ is finite and $\fF=\{X\}$, then $\fU$ being an ultrafilter means that $c$ only depends on one coordinate. In other words, every polymorphism of $K_3$ is essentially a projection.

    One can also regard $n$-ary polymorphisms as voting rules for an $n$-sized electorate: indeed, each individual voter chooses an element of the structure $H$ and the voting rule should determine a final outcome in $H$ based on the list of votes. For example, simple majority vote corresponds to a $K_2^3 \to K_2$ polymorphism. Keeping this in mind, polymorphisms of $K_3$ being essentially projections is a version of Arrow's paradox, namely, the fact that voting systems for selecting one of three candidates satisfying some natural requirements are necessarily dictatorships, i.e., the outcome depends only on a single voter. 


	Now, let us discuss our proof of the second, less surprising statement of Theorem \ref{t:main}. We use a result of Barto-Opr\v{s}al-Pinsker \cite{barto2018wonderland} (which builds on the previous work of Bulatov-Jeavons-Krokhin \cite{bulatov2005classifying}) that allows us to
      define a constructive way of reduction between homomorphism problems of structures for which $(*)_{\mc{D}}$ fails. This reduction turns out to be essentially sufficient to our purposes as well, yielding the following:
	
	\begin{theorem}
		\label{c:hardness} If $\lnot (*)_\mc{D}$, then $K_\mc{D} \iff K_{K_3} \iff  UL$.
	\end{theorem}
	
	On the course of proving this statement, one can isolate a new notion of reducibility, which is interesting on its own.
	
	\begin{definition}
		\label{d:finitely_reduces} Let $\fD$ and $\fE$ be finite structures. We say that the \emph{$\fE$-homomorphism problem finitely reduces to the $\fD$-homomorphism problem}, or in short, \textit{$\mc{E}$ finitely reduces to $\mc{D}$}, if there exist operations $\Gamma$, $\Phi$ and $\Psi$ such that: 
		
		(1) if $\fX$ is an instance of $\fE$, then $\Gamma(\fX)$ is an instance of $\fD$;
	
		(2) for every instance $\fX$ of $\fE$ the operation $\Phi$ maps $\fX\to\fE$ homomorphisms to $\Gamma(\fX)\to\fD$ homomorphisms and $\Psi$ maps $\Gamma(\fX)\to\fD$ homomorphisms to $\fX\to\fE$ homomorphisms;

        (3) if there exists a finite substructure $\fH$ of $\Gamma(\fX)$ that does not admit a homomorphism to $\fD$, then there exists a finite substructure $\fF$ of $\fX$ that does not admit a homomorphism to $\fE$.
	\end{definition}

    \begin{remark}
		\label{r:fin_red_transitive}
		Observe that finite reducibility is a transitive relation. 
	\end{remark}

 \begin{remark}
        \label{r:suff_for_fin_red}
        Assuming that (1) and (2) of Definition~\ref{d:finitely_reduces} are satisfied, to check (3) it suffices to verify the following. For every finite substructure $\fH$ of $\Gamma(\mc{X})$ there exists a finite substructure $\fF$ of $\fX$ such that there exists an $\fH\to \Gamma(\fF)$ homomorphism.
        \end{remark}

        \begin{proof}
        Assume that a finite substructure $\fH$ of $\Gamma(\fX)$ does not admit a homomorphism to $\fD$. Since there is a finite substructure $\fF$ of $\fX$ and an $\fH\to\Gamma(\fF)$ homomorphism, $\Gamma(\fF)$ cannot admit a homomorphism to $\fD$. Then, by (2) of Definition~\ref{d:finitely_reduces}, $\fF$ cannot admit a homomorphism to $\fE$.
        \end{proof}
	
	The definition of finite reducibility is tailored to suit the next statement. 
	
	\begin{proposition} 
		\label{pr:finitary_red} For finite structures $\fD$ and $\fE$, if the $\fE$-homomorphism problem finitely reduces to the $\fD$-homomorphism problem, then $K_\fD \implies K_\fE$.
	\end{proposition}
 
    \begin{proof}
		Assume that $\fE$ finitely reduces to $\fD$ and $K_\fD$ holds. Let $\fX$ be an instance of $\fE$, and suppose that every finite substructure $\fF$ of $\fX$ admits a homomorphism to $\fE$. Consider $\Gamma(\fX)$ (given by Definition~\ref{d:finitely_reduces}), which is an instance of $\fD$.

            By (3) of Definition~\ref{d:finitely_reduces}, every finite substructure of $\Gamma(\fX)$ admits a homomorphism to $\fD$.
            
		
		
		Thus, by $K_\fD$, there exists a $\Gamma(\mc{X})\to\fD$ homomorphism. Now (2) of Definition~\ref{d:finitely_reduces} provides an $\fX\to\fE$ homomorphism.
	\end{proof}

	Now, let us recall the notion of constructive reduction formulated in \cite{barto2018wonderland}.
 
        \textbf{Notation.} Let $\Sigma_\fD$ denote the signature of the structure $\fD$, i.e., the set of relations of $\fD$.

        \begin{definition}
        \label{d:pp_power}
		For finite structures $\fD$ and $\fE$ we say that \emph{$\fE$ is a $pp$-power of $\fD$} if for some $n\in\om$ we have $E=D^n$ and for every $k\in\om$ and relation symbol $R$ of arity $k$ in $\Sigma_\fE$ there exist $m_R\in\om$ and relation symbols $\alpha_{R,1},\ldots,\alpha_{R,m_R}$ in $\Sigma_\fD\cup\{=\}$ such that
        \begin{equation}\label{e:simp_def_witness_2}
        (\ovl z_1,\ldots,\ovl z_k)\in R^\fE\iff \exists \ovl w\ \bigwedge_{i=1}^{m_R} {\alpha_{R,i}}^\fD(\ovl z_1,\ldots,\ovl z_k,\ovl w).  
        \end{equation}
		
        \end{definition}

        \begin{definition}
        Two structures $\fD$ and $\fE$ are \emph{homomorphically equivalent} if there exist $\fD\to\fE$ and $\fE\to\fD$ homomorphisms.
        \end{definition}
	
	Now we can state our result.
 
	\begin{theorem}
		\label{t:fin_const_fin_red}
		For finite structures $\fD$ and $\fE$ we have the following:
  
            (A) If $\fD$ and $\fE$ are homomorphically equivalent, then they finitely reduce to each other.
            
            (B) If $\fE$ is a $pp$-power of $\fD$, then $\fE$ finitely reduces to $\fD$.
	\end{theorem}
 
    \begin{proof}
        See Section~\ref{s:hard} and the Appendix.
    \end{proof}

    To exploit this result, we use the following theorem which follows from \cite[Theorem~1.3]{barto2018wonderland}.
    

    \begin{theorem}
    \label{t:intract_constructs_all} 
    Assume that $\mc{D},\mc{E}$ are finite structures so that $\lnot (*)_\mc{D}$. Then $\mc{E}$ is homomorphically equivalent to a pp-power of $\mc{D}$.
    \end{theorem}

	Putting all these ingredients together we can prove Theorem \ref{c:hardness}.

	\begin{proof}[Proof of Theorem \ref{c:hardness}]
	   To see the first equivalence, by Proposition \ref{pr:finitary_red}, it suffices to prove that $K_\fD$ and $K_{K_3}$ finitely reduce to each other. By Theorems~\ref{t:fin_const_fin_red} and \ref{t:intract_constructs_all}, this follows from the fact that both $(*)_\mc{D}$ and $(*)_{K_3}$ fail. The second equivalence follows from Theorem \ref{t:K_K_3_UL}.
	\end{proof}

	\subsection{Easy problems} Second, let us give a high level overview of the main ideas of the proof of \eqref{c:maineasy} of Theorem \ref{t:main}. Call 
	an instance $\mc{X}$ of $\mc{D}$ \emph{finitely solvable} if every finite substructure of $\fX$ admits a homomorphism to $\mc{D}$. 
	
	In order to construct models in which, say $K_{K_3}$ fails, we use classical ideas of Mostowski and Fraenkel. They constructed so-called \emph{permutation models}, that is, models in which all of the axioms of ZF hold except for the Axiom of Extensionality (a.k.a. Equality): in addition to the $\emptyset$, there is a collection of \emph{atoms} which do not have elements. The corresponding axiom system is denoted by ZFA. While these models are not models of ZF, forcing arguments using similar ideas often -- this happens in our case as well -- yield ZF models with analogous properties (see, e.g., \cite[Chapter 15]{jech2003set}). 

    Let us describe how to construct a ZFA model. Build the cumulative hierarchy of sets, but this time starting not just from the $\emptyset$, but $
    \emptyset$ together with the collection of atoms, $A$. Fix some subgroup $\Gamma$ of the permutation group of atoms. Note that the action on atoms naturally extends to an action on the sets containing atoms, sets of sets of atoms, etc. (define  $\gamma \cdot x$ to be $\{\gamma \cdot y:y \in x\}$). Now consider only the collection, denoted by $HS$, of elements of the hierarchy which are hereditarily invariant (or hereditarily symmetric) under some further ``large" subgroups $\Gamma'<\Gamma$. In such a way we will ensure that the atoms are somewhat indistinguishable from each other within the model. It turns out that $HS$ is a model of ZFA. By choosing $A$ and $\Gamma$ appropriately, we will be able to control the truth value of the statements $K_{\mc{D}}$ in $HS$.

 	Now, for a fixed structure $\mc{D}$ with $(*)_\mc{D}$ we take a prime $p$ and a cyclic polymorphism $\phi$ of arity $p$ (see \ref{d:tractability}). We define a graph on the atoms, consisting of disjoint cycles of size $p$, and choose the group to be the one generated by rotations of single cycles. It will easily follow that such a graph cannot have a $3$-coloring in the model (i.e., an invariant one), but every finite subgraph has one, that is, $\lnot \cpcth$.
	
	In order to show that $\cpct{D}$ holds in this model, we have to prove that given a finitely solvable instance $\fX$ of the $\mc{D}$-homomorphism problem, it has a solution (i.e.,\ an $\fX \to \fD$ homomorphism) in the model. By the definition of the model, $\fX$ is invariant under some subgroup $\Gamma'$, and it suffices to construct a $\Gamma'$-invariant solution. The main observation is that this can be done starting from an arbitrary solution, outside the model, which exists by finite solvability and compactness (using AC). We can then use compactness (AC) again to find an invariant solution inside the model. Since we have compactness, it suffices to construct solutions that are invariant under any finite subset of $\Gamma'$. This can be done using cyclic polymorphisms: if $h_0$ is a homomorphism, $\alpha$ is a group element with $\alpha^p=1$, and $\phi$ is a cyclic polymorphism of $\fD$ of arity $p$, then $h=\phi(h_0,\alpha \cdot h_0, \dots, \alpha^{p-1} \cdot h_0)$ is a solution invariant under $\alpha$.
	
	To make this intuition precise, we proceed to describe a general theorem, which can be applied without familiarity with forcing or abstract set theory. Let $\Gamma$ be a group. A collection $\mc{F}$ of subgroups of $\Gamma$ is called a \emph{filter} if 
	\begin{itemize}
		\item $\{1\} \not \in \mathcal{F}$, $\Gamma \in \mathcal{F}$
		\item $\Delta, \Delta' \in \mathcal{F}$ implies $\Delta \cap \Delta' \in \mathcal{F}$,
		\item $\Delta \in \mathcal{F}$, $\Delta <\Delta'$ implies $\Delta' \in \mathcal{F}$, 
		\item $\Delta \in \mathcal{F}$, $\gamma \in \Gamma$ implies $\gamma^{-1} \Delta \gamma \in \mathcal{F}$. 
	\end{itemize}
	
	The main example of a filter to be kept in mind for a group acting on an infinite set $T$ is the collection of subgroups that contain the \emph{pointwise stabilizer} $\stab_{pw}(F)$ of some finite set $F \subset T$.
	
	If $\cdot: \Gamma \times X \to X$ is an action of $\Gamma$ on a set $X$ and $Y_i \subset X$ for $i \in I$, then the setwise \emph{stabilizer of $(Y_i)_{i \in I}$} is the subgroup \[\stab((Y_i)_{i \in I})=\{\gamma \in \Gamma: \forall i \in I \ \gamma \cdot Y_i=Y_i\}.\]

	The action $\cdot$ extends to $X^n$, $\bigcup_n X^n$ coordinate-wise. Similarly, if $D$ is any set on which $\Gamma$ does not act (or equivalently, assumed to act trivially), we can extend the action to $X \times D$. We also use $\cdot$ to denote these actions. Note that this allows us to talk about the stabilizers of structures on $X$ (as they are collections of subsets of $\bigcup_n X^n$) and functions $X \to D$, as they are subsets of $X \times D$. For $S \subset \Gamma$ and an action of $\Gamma$ on a set $Z$, a set $Y \subset Z$ is called \emph{$S$-invariant} if $S \subset \stab(Y)$.
	
	\begin{definition}
		Assume that $\Gamma$ acts on $X$, the universe of some structure $\mc{X}$. We say that $\mc{X}$ is \emph{$\mathcal{F}$-symmetric (w.r.t. the given action of $\Gamma$)} if $\stab(\mc{X}) \in \mathcal{F}$.   
	\end{definition}
	
	In Section~\ref{s:easy}, we prove the following theorem, which one can use as a black box, without familiarity with abstract set theory.
	
	\begin{theorem}
		\label{t:forcingblackbox}
		Let $\Gamma$ be a group and $\mathcal{F}$ be a filter of subgroups of $\Gamma$. Assume that for some structures $\mc{D},\mc{E}$ the below statements can be proved in ZFC:
		\begin{enumerate}
			\item For every finitely solvable $\mathcal{F}$-symmetric  $\mc{D}$-instance $\mathcal{X}$ there exists a $\mathcal{F}$-symmetric homomorphism from $\mathcal{X}$ to $\mc{D}$. 
			
			\label{c:forevery} 
			\item \label{c:example} There exists a finitely solvable $\mathcal{F}$-symmetric $\mc{E}$-instance $\mathcal{Y}$ with $\stab(F) \in \mathcal{F}$ for all finite $F \subset Y$, which does not admit an $\mathcal{F}$-symmetric homomorphism to $\mc{E}$.
		\end{enumerate}
		Then there exists a model of ZFA in which $\cpct{D}$ holds but $\cpct{E}$ fails.
	\end{theorem}

    \begin{proof}
        See Section~\ref{s:easy}.
    \end{proof}
	
	We immediately obtain a more combinatorial form of the theorem using a compactness argument.
	\begin{corollary}
		\label{c:finitecheck}
		Theorem \ref{t:forcingblackbox} holds even if we replace 
		\eqref{c:forevery} by
		
		\begin{enumerate}
			\item[(1')] For every  $\mathcal{F}$-symmetric finitely solvable $\mc{D}$-instance $\mathcal{X}$, there exists some $\Gamma' \in \mathcal{F}$ such that for every finite $S \subset \Gamma'$ and $\{x_1,\dots,x_n\} \subset X$ there exists an $S$-invariant partial homomorphism of $\restriction{\mc{X}}{\{x_1,\ldots,x_n\}}$ to $\mc{D}$.
		\end{enumerate}
		In particular, to check that \eqref{c:forevery} holds it suffices to find a subgroup $\Gamma' \in \mathcal{F}$ such that for all finite $S \subset \Gamma'$ there is an $S$-invariant homomorphism from $\mathcal{X}$ to $\mc{D}$.  
	\end{corollary}
	\begin{proof}[Proof of Corollary \ref{c:finitecheck}]
		We show that \eqref{c:forevery} of Theorem \ref{t:forcingblackbox} holds using a compactness argument. Take the compact topological space $D^X$, where $D$ is endowed with the discrete topology. The following subsets are closed in this space for all $\gamma \in \Gamma, x_i \in X,n \in \N$: $C_\gamma=\{h:h \text{ is $\gamma$-invariant}\}$, $C_{x_1,\dots,x_n}=\{h: \restriction{h}{\{x_1,\ldots,x_n\}} \text{ is a partial homomorphism}\}$. By our assumption, the intersection of any finite collection of sets of the form $(C_\gamma)_{\gamma\in \Gamma'}$ or $ (C_{x_1,\dots,x_n})_{x_i \in X}$ is nonempty, hence there is some $h$ in the intersection of all of these sets (note that we work in ZFC, so compactness can be used), and such an $h$ is a $\Gamma'$-invariant, and thereby $\mathcal{F}$-symmetric homomorphism $\mathcal{X} \to \mc{D}$. 
	\end{proof}
    This will directly yield our main result for ZFA models. 
    \begin{theorem}
		\label{c:ZFA}
		There exists a model of ZFA in which $\cpct{D}$ holds precisely if $(*)_\mc{D}$ does.
		
	\end{theorem}
 
	Finally, using forcing arguments, one can obtain a ZF result as well.
	
	\begin{theorem}
		\label{t:forcing} There exists a model of ZF in which $\cpct{\mc{D}}$ holds precisely if $(*)_{\mc{D}}$ does.
	\end{theorem}

   We finish with briefly outlining some of the ideas that we believe might be transferable to the finitary setting and have been already used by Cohen to prove the independence of AC (see \cite[Section 15]{jech2003set}). Our view of forcing is the following: start with a model $M$ of ZFC and a Boolean algebra $\mathcal{B}$, and build a model $M^\mc{B}$ where a statement of the form $x \in y$ instead of having a true/false value, has one in $\mathcal{B}$. Then, taking an appropriate ultrafilter in $\mathcal{B}$ one can ``collapse" the truth value to true/false and get a classical model $N$ extending $M$. 
   
    To mimic the behavior of ZFA models, we start with a Boolean algebra and a subgroup $\Gamma$ of its automorphisms. Now, the action of $\Gamma$ extends to $M^\mc{B}$, and we will consider only the collection $N'$ of elements of $N$, which collapse from elements of $M^\mc{B}$ that are sufficiently invariant under the action. $N'$ turns out to be a model of ZF. In this model, sets of generic reals, which are somewhat indistinguishable from within $M$, will play the role of atoms.
	
	\section{Hard problems are hard}
	\label{s:hard}
	In this section, we show Theorem \ref{t:fin_const_fin_red}:
	\begin{namedthm*}{Theorem \ref{t:fin_const_fin_red}}
		For finite structures $\fD$ and $\fE$ we have the following:
  
            (A) If $\fD$ and $\fE$ are homomorphically equivalent, then they finitely reduce to each other.
            
            (B) If $\fE$ is a $pp$-power of $\fD$, then $\fE$ finitely reduces to $\fD$.
	\end{namedthm*}
	
	\begin{proof}
		(A) Now the structures $\fD$ and $\fE$ are of the same signature and there are homomorphisms $\theta_1:\fD\to\fE$ and $\theta_2:\fE\to\fD$. Set $\Gamma(\fX)=\fX$, $\Phi(\varphi)=\theta_2\circ\varphi$ and $\Psi(\psi)=\theta_1\circ\psi$.
		This clearly works.
		
	(B) {\bf Sketch.} The precise argument is quite heavy on notation, so we give an informal outline. Details can be found in the Appendix.    

            By definition, we have $E=D^n$, and by (\ref{e:simp_def_witness_2}), $(\ovl z_1,\ldots,\ovl z_k)\in R^\fE$ comes with some number of witnesses in $D$.  Given an instance $\fX$ of $\fE$, we build the $\fD$-instance $\Gamma(\fX)$ by adding the following elements:
            \begin{itemize}
                \item $n$ ``formal coordinates'' for every $x \in X$;
                \item the appropriate number of ``formal witnesses of relation'' for every relation symbol $R$ in $\Sigma_{\fE}$ and tuple $\ovl x \in R^\fX$, arising from \eqref{e:simp_def_witness_2}.   
            \end{itemize}

            We must carefully manage using the equality sign in (\ref{e:simp_def_witness_2}). First, we can assume without loss of generality that no witness appears in any of the equalities. Second, the appropriate formal coordinates of those $x \in X$ that appear in relations $R^{\fX}$ whose $pp$-definitions include equalities need to be identified. (Formally, we quotient out by the generated equivalence relation.)   

            We need to transfer homomorphisms (to $\fE$ and $\fD$) between $\fX$ and $\Gamma(\fX)$. On the one hand, a homomorphism $\varphi: \fX \to \fE$  gives rise to a map $\Phi(\varphi):\Gamma(\fX) \to \fD$ in a straightforward way, sending formal coordinates of $x \in X$ to actual coordinates of $\varphi(x)$ (and by necessity this factors through the quotient map). 
            The relations on $\Gamma(\fX)$ are tailored to ensure that $\Phi(\varphi)$ is a homomorphism.
            
            On the other hand, if $\psi:\Gamma(\fX) \to \fD$ is a homomorphism, evaluating $\psi$ on all formal coordinates of each $x \in X$ gives rise to a map $\Psi(\psi): \fX \to D^n=E$. The images of formal witnesses serve as actual witnesses in $D$, and the identification of the necessary formal coordinates ensures that all equalities in (\ref{e:simp_def_witness_2}) are satisfied. The two together imply that $\Psi(\psi)$ is a homomorphism to $\fE$. 

            We also need to check (3) of Definition~\ref{d:finitely_reduces}, regarding the finite substructures. We do this using Remark \ref{r:suff_for_fin_red}, by collecting necessary elements of $X$ into a finite set $F$ to provide enough relations in $\Gamma(\fF)$ for the images of all related tuples of $\fH$. Again, we have to take into account the identification of formal coordinates, which forces us to add extra elements to $F$ to ensure that all the necessary identifications already happen in $\Gamma(\fF)$.
            
            We make choices at multiple points along the construction but these are always possible in ZF because we always choose either finitely many times or from some fixed finite set.
	\end{proof}
	Let us remark that the first version of this manuscript used the technique developed in \cite{riley} to deal with the case of equality. However, Thornton pointed out to us that this is not necessary.
	
	\section{Easy problems are easy}
	
	\label{s:easy}
	In this section, we show Theorem \ref{t:forcingblackbox} and the remaining part of our main result, Theorem \ref{t:main}.
	
	\subsection{ZFA results.} 
	
	First, as a warm-up, we prove the following general theorem about ZFA models.
		\begin{manualtheorem}{\ref{t:forcingblackbox}}
			Let $\Gamma$ be a group and $\mathcal{F}$ be a filter of subgroups of $\Gamma$. Assume that for some structures $\mc{D},\mc{E}$ the below statements can be proved in ZFC:
		\begin{enumerate}
			\item For every finitely solvable $\mathcal{F}$-symmetric  $\mc{D}$-instance $\mathcal{X}$ there exists a $\mathcal{F}$-symmetric homomorphism from $\mathcal{X}$ to $\mc{D}$. 
			
			\label{c:forevery1} 
			\item \label{c:example1} There exists a finitely solvable $\mathcal{F}$-symmetric $\mc{E}$-instance $\mathcal{Y}$ with $\stab(F) \in \mathcal{F}$ for all finite $F \subset Y$, which does not admit an $\mathcal{F}$-symmetric homomorphism to $\mc{E}$.
		\end{enumerate}
		Then there exists a model of ZFA in which $\cpct{D}$ holds, but $\cpct{E}$ fails. In fact, the model depends only on $\Gamma$, $\mathcal{F}$ and $\mathcal{Y}$. 
	\end{manualtheorem}

	\begin{proof}
		Let $\mc{Y}$ be the structure from \eqref{c:example1}.
		
		We build a permutation submodel of the universe $V$ as in \cite[15.48]{jech2003set}. Let the set of atoms $Y'$ be chosen so that there is a bijection $b:Y' \to Y$ and define $\mathcal{Y}'$ to be a structure on $Y'$ by pulling back the relations form $\mc{Y}$. Similarly, define the $\Gamma$ action on $Y'$ by $\gamma \cdot y'=b^{-1}(\gamma \cdot b(y')).$
		
		Let $U$ be the permutation model corresponding to the set of atoms $Y'$ and the action $\cdot$ defined above. By the $\mathcal{F}$-symmetricity of $\mc{Y}$, $\mc{Y}' \in U$, and by the nonexistence of an $\mathcal{F}$-symmetric homomorphism to $\mc{E}$, the $\mc{E}$-instance $\mc{Y}'$ is not solvable in $U$. As the stabilizer of every finite set $F \subset Y'$ is in the filter, every partial map $Y' \to E$ with a finite domain is in $U$, in particular, $\mc{Y}'$ is finitely solvable. Thus, $U \models \lnot K_\mc{E}$.  
		
		Now, let $\mc{X} \in U$ be a finitely solvable $\mc{D}$-instance. We check that it is such an $\mathcal{F}$-symmetric instance in $V$ as well. Indeed, $\mc{X}$ is hereditarily $\mathcal{F}$-symmetric in $V$. To see that it is finitely solvable, just note that as $\mathcal{F}$ is a filter, any finite set containing only hereditarily $\mathcal{F}$-symmetric elements is in $U$, so every finite substructure of $\mc{X}$ in $V$ is also in $U$.
		
		By our assumptions, it admits an $\mathcal{F}$-symmetric homomorphism $h$ to $\mc{D}$ in $V$. But, as all elements of $X$ are hereditarily symmetric, so are all elements of $h$. Thus, $h$ is hereditarily symmetric as well, in particular, $h \in U$. 		
	\end{proof}
	
	Now we apply the above theorem in our particular case. 
	\begin{theorem}
		\label{c:ZFAa}
		There exists a model of ZFA in which $\cpct{D}$ if $(*)_\mc{D}$ does, and $K_{\mc{K}_3}$ fails.
		
	\end{theorem}
	Let us remark that a careful examination of the constructions presented in Section 
	\ref{s:hard} show that in the model below, $K_\mc{D}$ fails for all $\mc{D}$ with $\lnot(*)_\mc{D}$.
	
	The model is a straightforward modification of the model $\mc{N}2^*(3)$ from \cite{howard1998consequences}, which one could call $\mc{N}2^*(\text{Prime})$.		\begin{definition}
		\label{d:model}
		Let $\Gamma=\oplus_{p \text{ prime}} \mathbb{Z}_p$, with the standard generating set $(\gamma_{p_i})_{p_i}$. Then $\Gamma$ acts on the set \[Y=\bigcup_{\text{$p_i$ is the $i$th prime}}\{i\} \times \{0,\dots,p_{i}-1\},\] by $\gamma_{p_i} \cdot (i,j)=(i,j+1 \mod p_{i})$ and fixing every other element of $Y$. Let $\mathcal{Y}$ be the Schreier graph of $\Gamma$'s action on $Y$ (w.r.t. the generating set $(\gamma_p)_p$).
		Let $\mathcal{F}$ be the filter generated by the subgroups 
		$\langle \gamma_{p_i}:i \geq n \rangle$ for $n \in \N$.
		
	\end{definition}
	
	\begin{proof}
		We apply Theorem \ref{t:forcingblackbox}, using  the group $\Gamma$ defined above.  
		
		\textit{First}, we show that the ``in particular" part of Corollary \ref{c:finitecheck} holds, thereby guaranteeing \eqref{c:forevery} of Theorem \ref{t:forcingblackbox}. 
		
		Let $\mc{D}$ be a structure with $(*)_\mc{D}$ and assume that $\mathcal{X}$ is an $\mathcal{F}$-symmetric finitely solvable instance of $\mc{D}$. Let $\Gamma'$ witness that $\mathcal{X}$ is $\mc{F}$-symmetric, fix an $n \in \N$, a sequence $(\phi_p)_{p \geq p_n}$  of cyclic polymorphisms of arity $p$ as in Definition \ref{d:tractability}. We start with an easy observation. 
		
		\begin{claim}
			\label{cl:obvious}
			Let $\gamma \in \Gamma'$ with $\gamma^p=1$ and $h_0:\mathcal{X} \to \mc{D}$ be a homomorphism.
			Then the map \[h(x)=\phi_p(h_0(x),h_0(\gamma^{-1} \cdot x), h_0(\gamma^{-2}\cdot x), \dots,h_0 (\gamma^{-p+1}\cdot x))\] is a $\{\gamma\}$-invariant homomorphism. Moreover, if $h_0$ is $\{\gamma'\}$-invariant for some $\gamma' \in \Gamma'$, then so is $h$.
			
		\end{claim}
		\begin{proof}
			First note that by the $\Gamma'$-invariance of $\mc{X}$, for any $\delta \in \Gamma'$ the map $x \mapsto h_0(\delta^{-1} \cdot x)$ is a homomorphism.
			Thus, $h$ is a homomorphism, and by the cyclicity of $\phi_p$ we get 
			\[(\gamma \cdot h)(x)=h(\gamma^{-1} \cdot x)=\]\[\phi_p(h_0(\gamma^{-1} \cdot x),h_0(\gamma^{-2} \cdot x),\dots, h_0(\gamma^{-p} \cdot x))=\] 
			\[\phi_p(h_0(\gamma^{-p} \cdot x),h_0(\gamma^{-1} \cdot x),\dots, h_0(\gamma^{-p+1} \cdot x))=h(x).\] 
			
			In order to check the $\{\gamma'\}$-invariance, just observe that $\Gamma'$ is abelian, hence
			\[(\gamma' \cdot h)(x)= h(\gamma'^{-1} \cdot x)=\]
			\[\phi_p(h_0(\gamma'^{-1} \cdot x),h_0(\gamma^{-1}\cdot \gamma'^{-1} \cdot x),  \dots,h_0 (\gamma^{-p+1}\cdot \gamma'^{-1} \cdot x))=\]
			\[\phi_p(h_0(\gamma'^{-1} \cdot x),h_0(\gamma'^{-1}\cdot \gamma^{-1} \cdot x),  \dots,h_0 (\gamma'^{-1}\cdot \gamma^{-p+1}\cdot x))=\]
			\[\phi_p(h_0(x),h_0(\gamma^{-1}(x)), h_0(\gamma^{-2}\cdot x), \dots,h_0 (\gamma^{-p+1}\cdot x))=h(x).\]
		\end{proof}
		Let \[\Gamma''=\Gamma' \cap \stab_{pw}(\{(i,j): i \leq n\}),\]
		which is in $\mathcal{F}$, by the definition of $\mathcal{F}$. Let $S \subset \Gamma''$ be finite.	We have to show that there is an $S$-invariant homomorphism $h:\mathcal{X} \to \mc{D}$. Pick a subsequence sequence $(\gamma_{p_i})_{n \leq i \leq k}$  such that $S \subset \langle \gamma_{p_i}:n \leq i \leq k\rangle$. Since $\mathcal{X}$ is finitely solvable, there is a homomorphism $h_0:\mathcal{X} \to \mc{D}$. Applying Claim \ref{cl:obvious} inductively to $(\gamma_{p_i}), (\phi_{p_i})_{n \leq i \leq k}$ starting from $h_0$, we get a homomorphism invariant under each $(\gamma_{p_i})_{n \leq i \leq k}$, in turn under all elements of $S$. This shows that Corollary \ref{c:finitecheck} holds which implies the first part of Theorem \ref{t:forcingblackbox}.
		
		\textit{Second}, we check \eqref{c:example} of Theorem \ref{t:forcingblackbox} for $\mc{E}=K_3$, i.e., $3$-coloring. 
		
		Let $\mathcal{Y}$ be graph from Definition \ref{d:model}, it is clear that $\mathcal{Y}$ is $\Gamma$-invariant. Moreover, $\Gamma$ is the vertex disjoint union of cycles, hence it admits a $3$-coloring. Now, if $c$ was an $\mathcal{F}$-symmetric $3$-coloring, then we could find a finite set $F \subset Y$ such that for every $\gamma \in \stab_{pw}(F)$ we had $\gamma \cdot c=c$. In particular, for every large enough $p_i$ we had $\gamma_{p_i} \cdot c=c$, which means that $c$ is constant on the cycle corresponding to $p_i$, a contradiction.
		
		To finish the proof of Theorem \ref{c:ZFAa} observe that we used the same $\Gamma$, $\mathcal{F}$ and $\mc{Y}$ for all $\mc{D}$, hence in the model guaranteed by Theorem \ref{t:forcingblackbox} the statement $K_\mc{D}$ holds if $(*)_\mc{D}$ is true. 
	\end{proof}

	\subsection{ZF results.}
	Now we are ready to prove the main result of the paper:
	
	\begin{manualtheorem}{\ref{t:forcing}}
		There is a model of ZF, in which $\cpct{D}$ holds precisely when $(*)_\mc{D}$ does.
	\end{manualtheorem}

	The argument below is the forcing version of the one presented in Corollary \ref{c:ZFA}.
	
	\begin{proof}[Proof of Theorem \ref{t:forcing}]
		We build a symmetric submodel of a generic extension as in \cite[Chapter 15]{jech2003set}.	In order to make the forcing argument work, we need to consider a slightly modified version of the group from Definition \ref{d:model}.	Let $\mc{Y}$ be the graph from Definition \ref{d:model}, define $\Gamma_{i}=\oplus_{n \in \N} \mathbb{Z}_{p_i}$ the infinite fold direct sum of the group $\mathbb{Z}_{p_i}$ with itself, and $\Delta=\oplus_i \Gamma_i$. In other words, $\Delta$ is the sum of $\mathbb{Z}_p$s, where each cyclic group appears infinitely often. Fix an enumeration of the generators $(\gamma_{p_i,n})_{n \in \N, p_i \text{ prime}}$, and let $\mathcal{H}$ be the filter generated by the subgroups $\langle \gamma_{p_i,n}:n \in \N, i \geq l\rangle$, for $l \in \N$. 
		
		To each element of $Y$, we will add a countable set of Cohen reals on which sets $\Delta$ is going to act. To achieve this, we define an action of $\Delta$ on $Y \times \N \times \N$ as follows. For each $i$, fix a bijection $b_i:\N \to \Gamma_i$.
		A generator $\gamma_{p_i,n}$ acts on the $Y$ coordinate as in Definition \ref{d:model}, while for each $Y \ni v=(i,j)$ the corresponding first $\N$ is identified with $\Gamma_i$ via $b_i$, and there $\gamma_{p_i,n}$ acts by translation. Formally, for every $\gamma_{p_i,n}$, let  			
		\[\gamma_{p_i,n}\cdot ((i,j),m,l):=((i,j+1 \mod p_i),b_i^{-1}(\gamma_{p_i,n}b_i(m)),l),\]
		and let $\gamma_{p_i,n}$ fix every other element.
		
		Let us collect the most important observations about this action.

		\begin{claim} 
			\label{cl:action}
			Let $F \subset Y \times \N \times \N$ be finite. Then
			\begin{enumerate}
				\item $\stab_{pw}(F) \in \mathcal{H}$.
				\item $\forall i  \ \exists n \ \forall ((i,j),m,l) \in F \ \gamma_{p_i,n} \cdot ((i,j),m,l) \not \in F$. 
				\item $\forall i, n,m \ $ \[\gamma_{p_i,n} \cdot \{((i,j),m,l):l \in \N\}=\{((i,j+1),b_i^{-1}(\gamma_{p_{i},n}b_i(m)),l):l \in \N\}.\]
			\end{enumerate}
		\end{claim}
		\begin{proof}
			The first and last statements are clear. For the second one, define a partial map $b'_i((i,j),m,l)=b_i(m)$, and let $n$ be such that $\gamma_{p_{i},n}$ does not appear in $b'_i(F)$, that is, the support of every element of $b'_i(F)$ is disjoint from the copy of $\mathbb{Z}_{p_i}$ corresponding to $\gamma_{p_{i},n}$. Then $\gamma_{p_i,n}$ appears in every element of $b'_i (\gamma_{p_i,n} \cdot F)$, in particular, there is no $((i,j),m,l) \in F$ with $\gamma_{p_i,n}\cdot ((i,j),m,l) \in F$.		  
		\end{proof}

		Define $P$ to be the forcing notion consisting of partial functions of $Y \times \N \times \N \to 2$ with finite domain ordered by reverse inclusion. The $\Delta$ action on $Y \times \N$ gives rise to a $\Delta$ action on $P$. 
		\begin{claim}
			\label{cl:action2}
			For every $q \in P$ 
			\begin{enumerate}
				\item $\stab_{pw}(q) \in \mathcal{H}$
				\item $\forall i \ \exists n$ such $\gamma_{p_i,n} \cdot q$ and $q$ are compatible.
			\end{enumerate}
		\end{claim}
		\begin{proof}
			Applying Claim \ref{cl:action} to $\dom(q)$ yields the desired conclusions. 
		\end{proof}
		
		Let $B(P)$ be the complete Boolean algebra to which $P$ densely embeds (see \cite[Corollary 14.12]{jech2003set}). We will identify $P$ with its copy in $B(P)$. Let $M$ be the ground model. The action of $\Delta$ on $P$ gives rise to an action on $B(P)$ and on the collection of names $M^{B(P)}$ (see \cite[14.38]{jech2003set}). 
		
		Let $HS$ be the collection of hereditarily $\mathcal{H}$-symmetric names, $G$ be a generic filter on $P$, and $M[G]$ be the forcing extension. Define \[N=\{\dot{x}^G:\dot{x} \in HS\}.\]
		Then $N \subset M[G]$ is a model of ZF by \cite[Lemma 15.51]{jech2003set}. The most important technical tool to be used is the Symmetry Lemma \cite[14.37]{jech2003set}: for every $\gamma \in \Delta$, $q \in B(P)$, formula $\phi$ and names $\dot{x}_0,\dots,\dot{x}_{k-1}$ we have			
		\[q \Vdash \phi(\dot{x}_0,\dots,\dot{x}_{k-1}) \iff \gamma \cdot q \Vdash \phi(\gamma \cdot \dot{x}_0,\dots,\gamma \cdot \dot{x}_{k-1}).\]

		Let us first show the easier part of the theorem.
		
		\begin{proposition}
			\label{p:fails}
			$K_{\mc{K}_3}$ fails in $N$. In particular, by Corollary \ref{c:hardness}, if $\lnot (*)_\mc{E}$ then $N \models \lnot K_{\mc{E}}$.
		\end{proposition}
		\begin{proof}
			The poset $P$ adds a collection of Cohen reals indexed by $Y \times \N$, let $(\dot{r}_{(v,m)})_{v,m}$ be names for these. Let $\dot{A}_v$ contain all these names, i.e., $\dot{A}_v(\dot{r}_{v,m})=1,$
			and no other name appears in the domain of $\dot{A}_v$. Observe that by Claim \ref{cl:action} for all $\gamma_{p_i,n}$ and $v=(i',j)$ we have 
			$\gamma_{p_i,n} \cdot \dot{A}_{i',j}= \dot{A}_{i',j+1 \mod p_i}$ if $i=i'$ and  $\gamma_{p_i,n} \cdot \dot{A}_{i',j}= \dot{A}_{i',j}$ otherwise. In particular, all names $\dot{A}_v$ are in $HS$. These will play the role of atoms in $N$.

			Let $\dot{\mc{Y}}'$ be the name for a graph on the vertices $\dot{A}_v$, where $\dot{A}_v$ and $\dot{A}_{v'}$ are adjacent precisely if $v$ and $v'$ are in $\mc{Y}$. Note that $\dot{\mc{Y}}'$ is in $HS$, thus, $\dot{\mc{Y}}'^G$ is in $N$. Clearly, $\dot{\mc{Y}}'^G$ is isomorphic to $\mc{Y}$, therefore, every finite subgraph admits a $3$-coloring in $N$. 
			
			Now, assume that $\mc{Y}'$ admits a $3$-coloring in $N$. Then, for some $q \in G$ and $\dot{c} \in HS$ we have \[q \Vdash \dot{c} \text{ is a homomorphism } \dot{\mc{Y}}' \to K_3.\] Therefore, there is an $i$ such that for each $\gamma_{p_i,n}$ we have $\gamma_{p_i,n} \cdot \dot{c}=\dot{c}$ and $\gamma_{p_i,n} \cdot  q=q$. Also, there exists some $q' \leq q$ with $q' \in P$ and $d \in D$ such that 
			\[q' \Vdash \dot{c}(\dot{A}_{(i,0)})=d.\]
			By Claim \ref{cl:action2} there exists some $n$ with $\gamma_{p_i,n} \cdot q'$ compatible with $q'$. Then by the Symmetry Lemma
			\[\gamma_{p_i,n} \cdot q' \Vdash (\gamma_{p_i,n} \cdot\dot{c})(\gamma_{p_i,n} \cdot \dot{A}_{(i,0)})=d,\]
			hence, 
			\[q' \vee \gamma_{p,i,n} \cdot q' \Vdash \dot{c}(\dot{A}_{(i,0)})=d=\dot{c}(\dot{A}_{(i,1)}).\]
			As $\gamma_{p_i,n} \cdot q=q$, we have $q'  \vee \gamma_{p_i,n} \cdot q' \leq q$, contradicting that $q$ forces that $\dot{c}$ is a $3$-coloring.
		\end{proof}

		Now we turn to the proof of the more involved part:
		\begin{proposition}
			\label{pr:mainhard}
			If $(*)_\mc{D}$ holds then $N \models K_\mc{D}$.
		\end{proposition}
		
		\begin{proof}
			Fix a $\mathcal{D}$ with $(*)_\mc{D}$ and let $(\phi_p)_{p \geq p^*}$ be the collection of the corresponding cyclic polymorphisms. Assume that we are given a finitely solvable instance $\mc{X}$ in $N$. Then $\mc{X}$ is finitely solvable in $M[G]$, so as it is a model of ZFC, it admits a homomorphism to $\mc{D}$ in $M[G]$. 
			In particular, for some $q \in G $, name $\dot{\mc{X}} \in HS$ with $\dot{\mc{X}}^G=\mc{X}$ there is a name $\dot{h}$
			such that  \[q \Vdash \dot{h} \text{ is a homomorphism from $\dot{\mc{X}}$ to $\mc{D}$}.\] 
			
			The overall strategy is simple: using the cyclic polymorphisms and compactness, from $\dot{h}$ we construct a name $\in HS$ for a homomorphism. 
			
			As $q \in P$ and $\dot{\mc{X}} \in HS$ we can find an $i_0 \geq p^*$ large enough \[\langle \gamma_{p_i,n}:n \in \N,p_i \geq i_0 \rangle \subseteq \stab_{pw}(p) \cap \stab(\dot{\mc{X}}), \] and let $\Delta'=\langle \gamma_{p_i,n}:n \in \N,p_i \geq i_0 \rangle$.  
			
			First we show some technical lemmas. For $q \in B(P)$, $\gamma \in \Delta$ define $Homo_q$ to be the collection of names $\dot{h}$ for which $\dom(\dot{h}) \subseteq \dot{X},$
			\[q \Vdash \dot{h} \text{ is a homomorphism from $\dot{\mc{X}}$ to $\mc{D}$},\]
			and for $\gamma \in \Delta$ let $\inv_{q,\gamma}$ to be the collection of the ones $\dot{h}$ with
			\[q \Vdash \gamma \cdot \dot{h}=\dot{h}.\]
			
			Let us show now the forcing analogue of Claim \ref{cl:obvious}.
			
			\begin{lemma}
				\label{l:invariance}
				Assume that $\dot{h}_0 \in \homo_q$ 
				and $\gamma \in \Delta'$ with $\gamma^p=1$ with $p \geq p_{i_0}$ prime. Then there exists a name $\dot{h} \in \homo_q \cap \inv_{q,\gamma}$.  Moreover, if $\dot{h}_0  \in \inv_{q,\gamma'}$ then $\dot{h} \in \inv_{q,\gamma'}$ for any $\gamma' \in \Delta'$. 
			\end{lemma}	
			\begin{proof}
				By the Symmetry Lemma and the fact that $\gamma \in \stab(\dot{X})$, we have that  $\dot{h}_0 \in \homo_{p}$ is equivalent to
				$\gamma \cdot \dot{h}_0 \in \homo_{\gamma \cdot q}.$
				By $\gamma \in \stab_{pw}(q)$ these are also equivalent to $\gamma \cdot \dot{h}_0 \in \homo_{q}.$ Therefore, all names
				$\dot{h}_0,\gamma \cdot \dot{h}_0,\dots, \gamma^{p-1} \cdot \dot{h}_0$ are forced to be homomorphisms by $q$. Hence, 			
				\[q \Vdash \phi_p(\dot{h}_0,\gamma \cdot \dot{h}_0,\dots, \gamma^{p-1} \cdot \dot{h}_0) \text{ is a homomorphism from $\dot{\mc{X}}$ to $\mc{D}$}.\]
				
				Let $\dot{h}$ be a name for the composition $\phi_p(\dot{h}_0,\gamma \cdot \dot{h}_0,\dots, \gamma^{p-1} \cdot \dot{h}_0)$. Then $\dot{h} \in \homo_q$ and the calculations similar to the ones in Claim \ref{cl:obvious} show both remaining statements of the lemma:
				\[q \Vdash \dot{h}=\phi_p(\dot{h}_0,\gamma \cdot \dot{h}_0,\dots, \gamma^{p-1} \cdot \dot{h}_0)\] is equivalent to
				\[\gamma \cdot q \Vdash \gamma \cdot\dot{h}=\gamma \cdot\phi_p(\gamma \cdot\dot{h}_0,\gamma^2 \cdot \dot{h}_0,\dots, \gamma^{p} \cdot \dot{h}_0),\] which implies
				\[q \Vdash \gamma \cdot\dot{h}=\phi_p(\dot{h}_0,\gamma \cdot \dot{h}_0,\dots, \gamma^{p-1} \cdot \dot{h}_0)=\dot{h},\]
				by the cyclicity of $\phi_p$ and the $\gamma$-invariance of $q$. The moreover part follows from the fact that $\Delta$ is abelian.
				
			\end{proof}		
			Enumerate the generators in $\Delta'$ as $(\gamma_n)_{n \in \N}$. Observe that by the choice of $i_0$, every such $\gamma_n$ has some prime order $p \geq p^*$. Applying Lemma \ref{l:invariance} repeatedly, starting with $\dot{h}$, we obtain a sequence of names $(\dot{h}_n)_{n \in \N}$, with $\dot{h}_n \in \inv_{q,\gamma_i} \cap \homo_q$ for all $i<n$. 
			
			Using the compactness of $\mc{D}^\mc{X}$, in $M[G]$, we have that 
			\[M[G] \models \exists h \in \mc{D}^{\mc{X}} \ h \in \bigcap_n \overline{\{\dot{h}^G_i:i \geq n}\}.\]
			Hence there is some name $\dot{h}$ and $q' \leq q$ form $G$ with 
			\[q' \Vdash \dot{h} \in \bigcap_n \overline{\{\dot{h}_i:i \geq n\}}.\]
			\begin{lemma}
				\label{l:invarianceh}
				$\dot{h} \in \homo_{q'} \cap \inv_{q',\gamma}$, for all $\gamma \in\stab_{pw}(q') \cap \Delta'$. 
			\end{lemma}
			\begin{proof}
				We have to show that 
				\[q' \Vdash \dot{h} \text{ is a homomorphism from $\dot{\mc{X}}$ to $\mc{D}$}\]
				and that for all $\gamma \in \stab_{pw}(q') \cap \Delta'$, we have 
				\[q' \Vdash \gamma \cdot\dot{h}=\dot{h}.\] 
				The first statement is clear, since $q'$ forces that $\dot{h}$ is in the closure of a set of homomorphisms.
				
				In order to see the second statement, we need the next observation.
				
				\begin{claim} For any names $\dot{x},\dot{y}$ we have
					\[q' \Vdash \forall n_0  \ \exists n \geq n_0 \ ((\dot{x} \in \dot{h} \iff \dot{x} \in \dot{h}_n) \land (\dot{y} \in \dot{h} \iff \dot{y} \in \dot{h}_n)).\] 
				\end{claim}
				\begin{proof}
					Fix $\dot{x},\dot{y}$. For any $G' \ni q'$ generic by the definition of the product topology, we have that 
					\[M[G'] \vDash  \forall n_0  \ \exists n \geq n_0 \ ( (\dot{x}^{G'} \in \dot{h}^{G'} \iff \dot{x}^{G'} \in \dot{h}^{G'}_n) \land (\dot{y}^{G'} \in \dot{h}^{G'} \iff \dot{y}^{G'} \in \dot{h}^{G'}_n)),\]
                    so the claim follows.		
				\end{proof}
                Now we show that for every $\dot{x}$ we have $q' \Vdash (\dot{x} \in \gamma \cdot \dot{h} \iff \dot{x} \in \dot{h})$.
                By the definition of the action on names, and the facts that $\dot{h}_n \in \inv_{\gamma,q'}$ and $\gamma \cdot q'=q'$ we get
                \[q' \Vdash\dot{x} \in \dot{h}_n \iff  \dot{x} \in \gamma \cdot \dot{h}_n \iff  \gamma^{-1} \cdot \dot{x} \in \dot{h}_n.\]
                Since the construction of the sequence $(\dot{h}_n)$ has been done in $M$, we get that 
                \[q' \Vdash \forall n \ \dot{x} \in \dot{h}_n \iff \gamma^{-1}  \cdot \dot{x} \in \dot{h}_n.\]
                Applying the claim to $\dot{x}$ and $\gamma^{-1} \cdot \dot{x}$ we get
                \[q' \Vdash \forall n_0  \ \exists n \geq n_0 \ ((\dot{x} \in \dot{h} \iff \dot{x} \in \dot{h}_n) \land (\gamma^{-1}\cdot \dot{x} \in \dot{h} \iff \gamma^{-1}\cdot \dot{x}\in \dot{h}_n)).\] 
                Combining the last two relations, we get 
				\[q' \Vdash \dot{x} \in \dot{h} \iff \gamma^{-1} \cdot \dot{x}\in \dot{h},\]
				so, using the definition of the action and $\gamma \cdot q'=q'$ again we obtain 
                \[q' \Vdash \dot{x} \in \dot{h} \iff \dot{x} \in \gamma \cdot \dot{h}. \] Thus,
				$q' \Vdash \dot{h}=\gamma \cdot \dot{h}.$			\end{proof}

			Now we show that there exists a name $\dot{h'} \in HS\cap \homo_{q'}$.
			
			For each $\dot{x} \in \dom(\dot{h})$  define $\dot{h}'(\dot{x})$ to be $\dot{h}(\dot{x}) \wedge q'$. We claim that $\dot{h}'$ is a suitable choice. First, by Lemma \ref{l:invarianceh} it is clear that \[\dot{h}' \in \bigcap_{\gamma \in \Delta'\cap \stab_{pw}(q')} \inv_{q',\gamma} \cap \homo_{q'},\] since $q' \Vdash \dot{h}'=\dot{h}.$  Assume that for some name $\dot{x}$ we had $\dot{h}'(\dot{x}) \neq (\gamma \cdot \dot{h}')(\dot{x})$. Note that $(\gamma \cdot \dot{h}')(\dot{x})=(\gamma \cdot \dot{h}')(\gamma^{-1} \cdot\dot{x})\leq \gamma \cdot q'=q'$. Thus, there was a $q'' \leq q'$ with $q''\leq \dot{h}'(\dot{x})$ and $q'' \wedge (\gamma \cdot \dot{h}')(\dot{x})=0$, or $q''\leq \gamma \cdot \dot{h}'(\dot{x})$ and $q'' \wedge  \dot{h}'(\dot{x})=0$. But then 
			$q'' \Vdash \dot{x} \in \gamma \cdot \dot{h}'\setminus  \dot{h}' \cup \dot{h}' \setminus \gamma \cdot \dot{h}'$, contradicting that $\dot{h}' \in \inv_{q',\gamma}$. Hence, for each $\gamma \in \Delta' \cap \stab_{pw}(q') \in \mathcal{H}$ we have $\gamma \cdot \dot{h}'=\dot{h}'$, showing that $\dot{h}' \in HS$. 
			
			By $\dot{h}' \in HS$, we have $(\dot{h}')^G \in N$ and it follows from $q' \in G$ and $\dot{h}' \in \homo_{q'}$ that $(\dot{h}')^G$ is a homomorphism from $\dot{\mc{X}}^G$ to $\mc{D}$, concluding the proof of Proposition \ref{pr:mainhard}.
			
		\end{proof}
		Propositions \ref{pr:mainhard} and \ref{p:fails} yield Theorem \ref{t:forcing}.
	\end{proof}

	\section{Open problems}
	As mentioned above, finite reducibility yields a complexity hierarchy on homomorphism problems.  
	\begin{problem}
		Describe the hierarchy of finite reducibility.
	\end{problem}
	It would be also interesting to see the relationship of reduction to other well-known forms of reducibility.
	
	While our work detects a distinction between easy and hard problems, the difference between the strength of different easy problems is yet to be investigated. For example, let $K_{3LIN_2}$ stand for the compactness principle for systems of linear equations over $\mathbb{F}_2$. 
	\begin{problem}
		Is there a model of ZF in which $K_{K_2}$ holds but $K_{3LIN_2}$ does not?
	\end{problem}
	
	Symmetric models are the least sophisticated way to get independence results over ZF. Correspondingly, even very weak forms of choice fail in them. It would be extremely interesting to construct models where some choice principles remain true, but the same split between easy and hard problems can be detected. A deep theory has been developed in the past couple of years to construct models (see, e.g., \cite{larson2020geometric}).
	\begin{problem}
		Is there a model of ZF+DC in which $K_\mc{D}$ holds exactly when $(*)_\mc{D}$ is true?
	\end{problem}
	\noindent \textbf{Acknowledgements.} We would like to thank Lorenz Halbeisen for pointing out the similarity between L\'evy's and Mycielski's theorems with results obtained in \cite{riley,toden}. We are very grateful to Riley Thornton for suggesting the currently included proof of Theorem \ref{t:fin_const_fin_red}. We are also thankful to Amitayu Banerjee, Jan Greb\'ik, Paul Howard, Asaf Karagila, G\'abor Kun, Michael Pinsker, Eleftherios Tachtsis, and Jind\v{r}ich Zapletal for their useful comments and enlightening discussions.

	\bibliographystyle{abbrv}
	\bibliography{bibliography.bib}

\section*{Appendix}

\begin{proof}[Proof of Theorem \ref{t:fin_const_fin_red} (B)]

By definition, we have $E=D^n$, and we recall from (\ref{e:simp_def_witness_2}) that for every $k\in\om$ and relation symbol $R$ in $\Sigma_\fE$ there are relations symbols $\alpha_{R,i}$ in $\Sigma_\fD\cup\{=\}$ such that
        \begin{equation*}
        (\ovl z_1,\ldots,\ovl z_k)\in R^\fE \iff \exists \ovl w \bigwedge_{i=1}^{m_R}{\alpha_{R,i}}^\fD(\ovl z_1,\ldots,\ovl z_k,\ovl w).
        \end{equation*}
        Here the $\ovl z_j$ are $n$-tuples, and $\ovl w$ is an $r$-tuple of variables (where $r$ depends on $R$). Recall that the $\alpha_{R,i}$ are not necessarily $(kn+r)$-ary relations. (e.g., some of them may be the equality). We may assume that for every $s\leq k$, $j\leq n$, and $l,l'\leq r$ none of the $\alpha_{R,i}$ are of the form $\ovl z_s(j)=\ovl w(l)$ or $w(l)=w(l')$ because otherwise we could eliminate the variable $\ovl w(l)$.
        
        Let $\fX$ be an instance of $\fE$. We need to construct $\Gamma(\fX)$. First, let us define an instance $\wtilde\fX$ of $\fD$. Let $\wtilde X$ contain the following elements:
        \begin{itemize}
        \item an element $\ovl x(j)$ for every $x\in X$ and $j\leq n$;
        \item an element $\ovl w_{R,\ovl x}(j)$ for every relation symbol $R$ in $\Sigma_\fE$, tuple $\ovl x\in R^\fX$, and $j\leq r$, where $r$ is the arity of $\ovl w$ in (\ref{e:simp_def_witness_2}).
        \end{itemize}
        Now we define the relations. For every $k\in\om$, relation symbol $R\in\Sigma_\fE$ of arity $k$, tuple $(x_1,\ldots,x_k)\in R^\fX$, and index $i\leq m_R$, if $\alpha_{R,i}$ is not the equality and it depends (formally) exactly on the variables $\ovl z_{s_1}(j_1),\ldots,\ovl z_{s_l}(j_l),\ovl w(t_1),\ldots,\ovl w(t_v)$, then let
        $$\left(\ovl x_{s_1}(j_1),\ldots,\ovl x_{s_l}(j_l),\ovl w_{R,(x_1,\ldots,x_k)}(t_1),\ldots,\ovl w_{R,(x_1,\ldots,x_k)}(t_v)\right)$$
        be an element of ${\alpha_{R,i}}^{\wtilde\fX}$. Relation symbols not occurring this way are interpreted in $\wtilde \fX$ as empty relations.
        
        Let us define a graph $G_\fX$ on $\wtilde X$. For every relation $R^\fX$ on $\fX$ and tuple $(x_1,\ldots,x_k)\in R^\fX$ we join $\ovl x_s(j)$ and $\ovl x_t(l)$ with an edge if one of the relations $\alpha_{R,i}$ is the equality $\ovl z_s(j)=\ovl z_t(l)$. For $x,y\in \wtilde X$ let $x\sim_\fX y$ if and only if $x$ and $y$ are in the same connected component of $G_\fX$. This is an equivalence relation. Let $[x]$ denote the $\sim_\fX$-class of $x$. Note that elements of the form $w_{R,\ovl x}(j)$ are alone in their $\sim_\fX$-class.
        
        Now we define $\Gamma(\fX)$. Let $\dom(\Gamma(\fX))$ contain the following elements:
        \begin{itemize}
        \item the equivalence class $[\ovl x(j)]$ for each $x\in X$ and $j\leq n$,
        \item the elements of the form $\ovl w_{R,\ovl x}(j)$ from $\wtilde \fX$.
        \end{itemize}
        For every relation symbol $R$ in $\Sigma_\fE$ and $i\leq m_R$, if $\alpha_{R,i}$ is not the equality, then let
        ${\alpha_{R,i}}^{\Gamma(\fX)}$ be the following
        $$\left\{
        \begin{array}{ll}
            \left([\ovl x_{s_1}(j_1)],\ldots,[\ovl x_{s_l}(j_l)],\ovl w_{R,(x_1,\ldots,x_k)}(t_1),\ldots,\ovl w_{R,(x_1,\ldots,x_k)}(t_v)\right):\ \\
            \left(\ovl x_{s_1}(j_1),\ldots,\ovl x_{s_l}(j_l),\ovl w_{R,(x_1,\ldots,x_k)}(t_1),\ldots,\ovl w_{R,(x_1,\ldots,x_k)}(t_v)\right)\in{\alpha_{R,i}}^{\wtilde \fX}
        \end{array}
        \right\}.$$
        We interpret the remaining relation symbols of $\Sigma_\fD$ as empty relations. Clearly, (1) of Definition~\ref{d:finitely_reduces} is satisfied. To check (2) we need to transfer homomorphisms.
        
        Let $\varphi:\fX\to \fE$ be a homomorphism. First, let us define a map $\wtilde\varphi:\wtilde X\to\fD$. Let $u\in\wtilde X$.

        \textbf{Case 1.} We have $u=\ovl x(j)$ for some $x\in \fX$ and $j\leq n$. Then set $\wtilde\varphi(\ovl x(j))=\varphi(x)(j)$.

        \textbf{Case 2.} We have $u=\ovl w_{R,(x_1,\ldots,x_k)}(j)$ for some $R\in\Sigma_\fE$, $(x_1,\ldots,x_k)\in R^{\fX}$, and $j\leq r$. Since $\varphi$ is a homomorphism, we have $(\varphi(x_1),\ldots,\varphi(x_k))\in R^\fE$, hence we can choose a tuple $\ovl w_{R,(\varphi(x_1),\ldots,\varphi(x_k))}$ that witnesses this in (\ref{e:simp_def_witness_2}). Note that this choice can be made in ZF because $\fE$ is finite and $\ovl w_{R,(\varphi(x_1),\ldots,\varphi(x_k))}$ depends only on $(\varphi(x_1),\ldots,\varphi(x_k))$, not $(x_1,\ldots,x_k)$. Set $\wtilde \varphi(u)=\ovl w_{R,(\varphi(x_1),\ldots,\varphi(x_k))}(j)$.
        
        \textbf{Claim 1.} The map $\wtilde\varphi$ is $\sim_\fX$-invariant.

        Fix $\ovl x(j),\ovl y(l)\in \wtilde\fX$ with $\ovl x(j)\sim_\fX\ovl y(l)$. Then there is a path $\ovl x(j)=\ovl x_0(j_0), \ovl x_1(j_1),\ldots,\ovl x_n(j_n)=\ovl y(l)$ in $G_\fX$. By the definition of $G_\fX$, for each $0\leq q<n$ there is an arity $k$, a relation $R$ of arity $k$, a tuple $(x^*_1,\ldots,x^*_k)\in R^{\fX}$, an index $0\leq i<m_R$, and coordinates $s,t\leq k$ such that $ x_q=x^*_s$, $x_{q+1}=x^*_t$ and $\alpha_{R,i}$ is $\ovl z_s(j_q)=\ovl z_t(j_{q+1})$. Since $\varphi$ is a homomorphism, we have $(\varphi(x^*_1),\ldots,\varphi(x^*_k))\in R^\fE$. Then
        $$\wtilde\varphi(\ovl x_q(j_q))=\wtilde\varphi(\ovl x^*_s(j_q))=\varphi(x^*_s)(j_q)=$$
        $$=\varphi(x^*_t)(j_{q+1})=\wtilde\varphi(\ovl x^*_t(j_{q+1}))=\wtilde\varphi(\ovl x_{q+1}(j_{q+1})),$$
        where the third equality follows from (\ref{e:simp_def_witness_2}) and the fact that $\alpha_{R,i}$ is $\ovl z_s(j_q)=\ovl z_t(j_{q+1})$. Applying this to every $0\leq q<n$ we conclude that $\wtilde\varphi(\ovl x(j))=\wtilde\varphi(\ovl y(l))$, which proves Claim 1.
        
        Now we define the homomorphism $\Phi(\varphi):\Gamma(\fX)\to\fD$. Let $u\in\Gamma(\fX)$.
        
        \textbf{Case 1.} We have $u=[\ovl x(j)]$ for some $x\in\fX$ and $j\leq n$. Set $\Phi(\varphi)(u)=\wtilde\varphi(\ovl x(j))=\varphi(x)(j)$. This is well-defined by Claim 1. Alternatively, to see that \textit{we do not need to choose} a representative element from $[\ovl x(j)]$, we can set $\Phi(\varphi)(u)=\bigcup\wtilde\varphi([\ovl x(j)])$, that is, the element of the singleton $\wtilde\varphi([\ovl x(j)])$.
        
        \textbf{Case 2.} We have $u=\ovl w_{R,(x_1,\ldots,x_k)}$ for some $R\in\Sigma_\fE$, $(x_1,\ldots,x_k)\in R^\fX$, and $j\leq n$. Set $\Phi(\varphi)(u)=\wtilde\varphi(u)=\ovl w_{R,(\varphi(x_1),\ldots,\varphi(x_k))}(j)$
        
        We need to verify that $\Phi(\varphi)$ is a homomorphism. It suffices to check those relation symbols in $\Sigma_\fD$ that are of the form $\alpha_{R,i}$ since other relation symbols are interpreted in $\Gamma(\fX)$ as empty relations.

        The image of a tuple
        $$\left([\ovl x_{s_1}(j_1)],\ldots,[\ovl x_{s_l}(j_l)],\ovl w_{R,(x_1,\ldots,x_k)}(t_1),\ldots,\ovl w_{R,(x_1,\ldots,x_k)}(t_v)\right)\in{\alpha_{R,i}}^{\Gamma(\fX)}$$
        is, by definition,
        $$(\varphi(x_{s_1})(j_1),\ldots,\varphi(x_{s_l})(j_l),\ovl w_{R,(\varphi(x_1),\ldots,\varphi(x_k))}(t_1),\ldots,\ovl w_{R,(\varphi(x_1),\ldots,\varphi(x_k))}(t_v)),$$
        which is in ${\alpha_{R,i}}^\fD$ by the choice of the tuple $\ovl w_{R,(\varphi(x_1),\ldots,\varphi(x_k))}$.

        Now let $\psi:\Gamma(\fX)\to\fD$ be a homomorphism. Set $\Psi(\psi)(x)=(\psi([\ovl x(1)]),\ldots,\psi([\ovl x(n)]))$ for every $x\in \fX$. Let $R$ be a relation symbol in $\Sigma_\fE$, and $(x_1,\ldots,x_k)\in R^\fX$. Then, for every relation symbol of the form $\alpha_{R,i}$ that is not the equality the tuple
        \begin{equation}\label{e:rel_in_gamma_X}
            \left([\ovl x_{s_1}(j_1)],\ldots,[\ovl x_{s_l}(j_l)],\ovl w_{R,(x_1,\ldots,x_k)}(t_1),\ldots,\ovl w_{R,(x_1,\ldots,x_k)}(t_v)\right)\tag{$\star$}
        \end{equation}
        is in ${\alpha_{R,i}}^{\Gamma(\fX)}$ by the definition of ${\alpha_{R,i}}^{\Gamma(\fX)}$. If, on the other hand, $\alpha_{R,i}$ is of the form $\ovl z_s(j)=\ovl z_t(l)$ for some $s,t\leq k$ and $j,l\leq n$, then $\ovl x_s(j)$ and $\ovl x_t(l)$ are joined by an edge in $G_\fX$. In particular, $[\ovl x_s(j)]=[\ovl x_t(l)]$, hence (\ref{e:rel_in_gamma_X}) is in ${\alpha_{R,i}}^{\Gamma(\fX)}$ again. Since $\psi$ is a homomorphism, we obtain that
        $$(\psi([\ovl x_{s_1}(j_1)]),\ldots,\psi([\ovl x_{s_l}(j_l)]),\psi(\ovl w_{R,(x_1,\ldots,x_k)}(t_1)),\ldots,\psi(\ovl w_{R,(x_1,\ldots,x_k)}(t_v)))$$
        is in ${\alpha_{R,i}}^\fD$ for every $i\leq m_R$. (Here the indices $s_1,\ldots,s_l,j_1,\ldots,j_l, t_1,\ldots,t_v$ depend on $i$.)
        That is, $\psi(\ovl w_{R,(x_1,\ldots,x_k)}(1)),\ldots,\psi(\ovl w_{R,(x_1,\ldots,x_k)}(r))$ witness that $(\Psi(\psi)(x_1),\ldots,\Psi(\psi)(x_k))\in R^\fE$. 

        It remains to prove that the above constructed $\Gamma$, $\Phi$ and $\Psi$ satisfy (3) of Definition~\ref{d:finitely_reduces}. By Remark~\ref{r:suff_for_fin_red}, it suffices to show that for every finite substructure $\fH$ of $\Gamma(\fX)$ there is a finite substructure $\fF$ of $\fX$ such that there exists an $\fH\to\Gamma(\fF)$ homomorphism.

        Fix a finite substructure $\fH$ of $\Gamma(\fX)$. First, we define a finite set $F'\subseteq\wtilde X$. All relations of $\fH$ are of the form ${\alpha_{R,i}}^{\Gamma(\fX)}|_\fH$. For each $\alpha_{R,i}$ and $\left([\ovl x_{s_1}(j_1)],\ldots,[\ovl x_{s_l}(j_l)],\ovl w_{R,(x_1,\ldots,x_k)}(t_1),\ldots,\ovl w_{R,(x_1,\ldots,x_k)}(t_v)\right)\in{\alpha_{R,i}}^\fH$ we add elements to $F'$ as follows. If $v\geq 1$, then we add $\ovl x_1(1),\ldots,\ovl x_1(n),\ldots,\ovl x_k(1),\ldots,\ovl x_k(n)$ to $F'$ (we read them off from the index of $\ovl w_{R,(x_1,\ldots,x_k)}(t_1)$). If $v=0$, then we \emph{choose} elements $\ovl x_1(1),\ldots,\ovl x_1(n),\ldots,\ovl x_k(1),\ldots,\ovl x_k(n)$ such that $(x_1,\ldots,x_k)\in R^\fX$ and $\ovl x_{s_1}(j_1),\ldots,\ovl x_{s_l}(j_l)$ represent the classes $[\ovl x_{s_1}(j_1)],\ldots,[\ovl x_{s_l}(j_l)]$ respectively. 
        
        \textbf{Remark.} We need to treat these cases separately because whenever witnesses are present they correspond to specific representatives, and we need to pick the right representatives for them to be in relation with the witnesses. Reading the representatives off from the index of $\ovl w$ is not done to avoid Choice, it is done to ensure the tuple is related in $\wtilde \fX$.
        
        Now $F'$ has finitely many elements. Add finitely many more elements to $F'$ to ensure that for every $\ovl x(j),\ovl y(l)\in F'$ if $\ovl x(j)\sim_\fX \ovl y(l)$, then there is a path from $\ovl x(j)$ to $\ovl y(l)$ in $F'$. Thus, the definition of $F'$ involves choosing finitely many elements from $\fX$ and $\wtilde\fX$, which can be done in ZF.
        
        Finally, we define $F$. First, let $F$ contain every element $x\in\fX$ for which there is $j\leq n$ such that $x(j)\in F'$. Second, if two elements $\ovl x(j),\ovl y(l)\in F'$ are connected by an edge in $G_\fX$, we add finitely many more elements to $F$ to ensure that they are connected in $G_\fF$ as well (see the definition of $G_\fX$). This involves choosing finitely many elements from $\fX$.

        Now we define a homomorphism $\vartheta:\fH\to\Gamma(\fF)$. Let $u\in\fH$.


        \textbf{Case 1.} The element $u$ does not occur as a coordinate in any tuple in any of the relations ${\alpha_{R,i}}^\fH$. Then pick $\vartheta(u)\in\Gamma(\fF)$ arbitrarily. This involves choosing finitely many elements from $\Gamma(\fF)$.

        \textbf{Case 2.} For some relation symbol $\alpha_{R,i}$ the element $u$ occurs as a coordinate of a tuple
        $$\left([\ovl x_{s_1}(j_1)],\ldots,[\ovl x_{s_l}(j_l)],\ovl w_{R,(x_1,\ldots,x_k)}(t_1),\ldots,\ovl w_{R,(x_1,\ldots,x_k)}(t_v)\right)\in{\alpha_{R,i}}^\fH.$$
        \textbf{Subcase 2.1.} We have $v\geq 1$. If $u=\ovl w_{R,(x_1,\ldots,x_k)}(t_m)$ for some $m\leq v$, then set $\vartheta(u)=u$, which is an element of $\Gamma(\fF)$ since $(x_1,\ldots,x_k)\in R^\fX$. If $u=[\ovl x_{s_p}(j_p)]$ for some $p\leq l$, then set $\vartheta(u)=[\ovl x_{s_p}(j_p)]_\fF$, which is the $\sim_\fF$-class of $\ovl x_{s_p}(s_p)$. As before, this does not involve choice since we simply read off the representatives from the index of $\ovl w_{R,(x_1,\ldots,x_k)}$. The map $\vartheta$ is well-defined by the construction of $F$: $\sim_\fX$-equivalent elements in $F'$ are $\sim_\fF$-equivalent as well.

        \textbf{Subcase 2.2.} We have $v=0$. Then $u=[\ovl x_{s_p}(j_p)]$ for some $p\leq l$. Set $\vartheta(u)=[\ovl x_{s_p}(j_p)]_\fF$, where $\ovl x_{s_p}(j_p)$ is the representative already chosen in the definition of $F'$ in the case $v=0$. Again, this is well-defined because $\sim_\fX$-equivalent elements in $F'$ are $\sim_\fF$-equivalent as well.

        Now the image of a tuple
        $$\left([\ovl x_{s_1}(j_1)],\ldots,[\ovl x_{s_l}(j_l)],\ovl w_{R,(x_1,\ldots,x_k)}(t_1),\ldots,\ovl w_{R,(x_1,\ldots,x_k)}(t_v)\right)\in{\alpha_{R,i}}^\fH$$
        under $\vartheta$ is
        $$\left([\ovl x_{s_1}(j_1)]_\fF,\ldots,[\ovl x_{s_l}(j_l)]_\fF,\ovl w_{R,(x_1,\ldots,x_k)}(t_1),\ldots,\ovl w_{R,(x_1,\ldots,x_k)}(t_v)\right),$$
        which is in ${\alpha_{R,i}}^{\Gamma(\fF)}$ because of the definition of $\Gamma$ and the fact that $(x_1,\dots,x_k)\in R^\fF$. (If $v\geq 1$, then $(x_1,\ldots,x_k)$ comes from the index of $\ovl w_{R,(x_1,\ldots,x_k)}(t_1)$, otherwise it was chosen in the definition of $F'$.) Thus $\vartheta$ is a homomorphism.
        
        The proof of Theorem~\ref{t:fin_const_fin_red} is now complete.
    
\end{proof}
 
\end{document}